\documentclass[12pt]{amsart}

\usepackage{amsfonts, amsthm, amsmath}

\usepackage{rotating}

\usepackage{tikz}

\usepackage{graphics}

\usepackage{amssymb}

\usepackage{amscd}

\usepackage[latin2]{inputenc}

\usepackage{t1enc}

\usepackage[mathscr]{eucal}

\usepackage{indentfirst}

\usepackage{graphicx}

\usepackage{graphics}

\usepackage{pict2e}

\usepackage{mathrsfs}

\usepackage{enumerate}
\usepackage[pagebackref]{hyperref}
\usepackage{cite}
\usepackage{color}
\usepackage{epic}
\usepackage{hyperref} 
\usepackage{framed}
\usepackage{mathabx}
\usepackage{booktabs}
\usepackage{makecell}
\newcolumntype{V}{!{\vrule width 2pt}}

\numberwithin{equation}{section}
\def\blue{\textcolor{blue}}
\def\red{\textcolor{red}}

\theoremstyle{plain}

\newtheorem{theorem}{Theorem}[section]

\newtheorem{lemma}[theorem]{Lemma}

\newtheorem{corollary}[theorem]{Corollary}

\theoremstyle{definition}

\newtheorem{example}[theorem]{Example}

\newtheorem{conj}[theorem]{Conjecture}

\newtheorem{remark}[theorem]{Remark}

\newtheorem{?}[theorem]{Problem}

\newcommand{\N}{\mathbb{N}}
\newcommand{\Po}{\mathbb{P}}
\newcommand{\Z}{\mathbb{Z}}

\newcommand{\PP}{\mathcal{P}}
\newcommand{\B}{\mathcal{B}}
\newcommand{\C}{\mathcal{OC}}
\newcommand{\CC}{\mathcal{C}}
\newcommand{\Od}{\mathcal{O}}

\newcommand{\id}{\mathrm{id}}
\newcommand{\G}{\mathcal{G}}
\newcommand{\Hw}{\mathcal{H}}

\def\des{\mathsf{des}}
\def\asc{\mathsf{asc}}
\def\exc{\mathsf{exc}}
\def\vexc{\widetilde\exc}
\def\Peak{\mathrm{PEAK}}
\def\cPeak{\mathsf{cPEAK}}
\def\S{\mathfrak{S}}
\def\cdes{\mathsf{cdes}}
\def\casc{\mathsf{casc}}
\newcommand{\BP}{\mathcal{BP}}
\newcommand{\cBP}{c\mathcal{BP}}

\def\m{{\bf m}}

\begin{document}

\title[A decomposition of ballot permutations]{A decomposition of ballot permutations, pattern avoidance and Gessel walks}

\author[Z. Lin]{Zhicong Lin}
\address[Zhicong Lin]{Research Center for Mathematics and Interdisciplinary Sciences, Shandong University, Qingdao 266237, P.R. China}
\email{linz@sdu.edu.cn}

\author[D.G.L.~Wang]{David G.L. Wang}
\address[David G.L. Wang]{School of Mathematics and Statistics \& Beijing Key Laboratory on MCAACI, Beijing Institute of Technology, 102488 Beijing, R.R. China}
\email{glw@bit.edu.cn}

\author[T. Zhao]{Tongyuan Zhao}
\address[Tongyuan Zhao]{College of Science, China University of Petroleum, 102249 Beijing, P.R. China}
\email{zhaotongyuan@cup.edu.cn}

\date{\today}

\begin{abstract}

A permutation whose any prefix has no more descents than ascents is called a ballot permutation. 
In this paper, we present a decomposition of ballot permutations that enables us to 
construct a bijection between ballot permutations and odd order permutations, which proves a set-valued extension of a conjecture due to  Spiro using the statistic of peak values. This bijection also preserves  the neighbors of the largest letter in permutations and thus resolves a refinement of Spiro's conjecture proposed by Wang and Zhang. Our decomposition can be extended to well-labelled positive paths, a class of generalized ballot permutations arising from  polytope theory, that were enumerated by Bernardi,  Duplantier and Nadeau. 

We will also investigate the enumerative aspect of ballot permutations avoiding a single pattern of length $3$ and establish a connection between $213$-avoiding ballot permutations and Gessel walks.
\end{abstract}



\maketitle


\section{Introduction}\label{sec1: intro}

The integer sequence 
$$\{b_n\}_{n\geq0}=\{1, 1, 1, 3, 9, 45, 225, 1575, 11025, 99225,\ldots\}$$
defined by the exponential generating function 
$$
\sum_{n\geq0}b_n\frac{z^n}{n!}=\sqrt{\frac{1+z}{1-z}}
$$
has several algebraic or combinatorial interpretations (see~\cite[A000246]{oeis}), among which are two classes of restricted permutations:
\begin{itemize}
\item permutations  whose any prefix has no more descents than ascents, called {\em ballot permutations};
\item permutations which are the products of cycles with odd lengths, called {\em odd order permutations}.
\end{itemize}
Denote by $\B_n$ and $\Od_n$ the set of ballot permutations and odd order permutations of length $n$, respectively. For instance, 
\begin{align*}
\B_4&=\{1234,1243,1324,1342,1423,2314,2341,2413,3412\}\quad\text{and}\\
\Od_4&=\{\id, (1)(234),(1)(243),(134)(2),(143)(2),(124)(3),(142)(3),(123)(4),(132)(4)\}.
\end{align*}
The fact that $|\Od_n|=b_n$ is an application of the combinatorial Exponential Formula~\cite[Corollary~5.1.6]{st2}, while $|\B_n|=b_n$ was proved by Bernardi,  Duplantier and Nadeau~\cite{BDN}. 

Recently, Spiro~\cite{spi} found a refinement of $|\B_n|=|\Od_n|$. Let $\S_n$ be the set of permutations of $[n]:=\{1,2,\ldots,n\}$. For each $\pi=\pi_1\pi_2\cdots\pi_n\in\S_n$ with $\pi_i=\pi(i)$, define 
\begin{align*}
\asc(\pi)&=|\{i\in[n-1]: \pi_i<\pi_{i+1}\}|,\\
\des(\pi)&=|\{i\in[n-1]: \pi_i>\pi_{i+1}\}| \quad\text{and}\\
\exc(\pi)&=|\{i\in[n-1]: \pi_i>i\}|,
\end{align*}
called the {\em ascent number}, the {\em descent number} and the {\em excedance number} of $\pi$, respectively. The distribution of descent number or excedance number over permutations interprets the classical {\em Eulerian polynomials} $A_n(t)$ (see~\cite[Chapter~1]{st1}), which can be defined by 
$$
\frac{A_n(t)}{(1-t)^{n+1}}=\sum_{k\geq0}(k+1)^nt^k.
$$
In other words, the Eulerian polynomial $A_n(t)$ equals
\begin{equation}\label{euler}
\sum_{\pi\in\S_n}t^{\des(\pi)}=\sum_{\pi\in\S_n}t^{\exc(\pi)}. 
\end{equation}
Interestingly, Spiro introduced a variation of excedance numbers, whose distribution over $\Od_n$ has the same distribution as descent numbers over $\B_n$. For a cycle $C=(c_1c_2\cdots c_k)$, let 
$$
\casc(C):=|\{i\in[k]:c_i<c_{i+1}\text{ where $c_{k+1}=c_1$}\}| \quad\text{and}\quad \cdes(C):=k-\casc(C)
$$
be the {\em cyclic ascent number} and the {\em cyclic descent number} of $c$, respectively. Note that  $\exc(\pi)$ can be defined as
$$
\exc(\pi)=\sum_C \casc(C),
$$
 where the sum runs over all cycles $C$  of $\pi$.  Introduce the variation of excedance number 
$$
\vexc(\pi)=\sum_C \min(\casc(C),\cdes(C)),
$$
 where the sum runs over all cycles $C$  of $\pi$. For example, if $\pi = (1, 8, 5, 6, 4)(2, 3, 9)(7)$ written as product of cycles, then $\vexc(\pi)=2+1+0=3$. The following conjecture was posed by Spiro~\cite[Conjecture~1.2]{spi} with several important special cases verified. 
 
 \begin{conj}[Spiro]\label{conj:Spiro}
 For $n\geq1$,
 \begin{equation}\label{Spiro}
\sum_{\pi\in\B_n}t^{\des(\pi)}=\sum_{\pi\in\Od_n}t^{\vexc(\pi)}. 
\end{equation}
 \end{conj}

Wang and Zhang~\cite{WZ} proposed a refinement of Spiro's conjecture by tracking the neighbors of the largest letter in permutations. For $k\in[n]$ and $\pi\in\S_n$, 
\begin{itemize}
\item if  $1<\pi^{-1}(k)=\ell<n$, then the letters $\pi(\ell-1)$ and $\pi(\ell+1)$ are called the {\em neighbors of $k$} in $\pi$;
\item if $\pi(k)\neq k$, then the letters $\pi^{-1}(k)$ and $\pi(k)$ are called the {\em cyclic neighbors of $k$} in $\pi$. 
\end{itemize}

\begin{conj}[Wang and Zhang]\label{conj:WZ}
For any letters $i\neq j$, let $\B_n(i,j)$ (resp.~$\Od_n(i,j)$) be the set of elements $\pi\in\B_n$ (resp.~$\pi\in\Od_n$) such that $i$ and $j$ are neighbors (resp.~cyclic neighbors) of $n$ in $\pi$. Then, 
 \begin{equation}\label{eq:WZ}
\sum_{\pi\in\B_n(i,j)}t^{\des(\pi)}=\sum_{\pi\in\Od_n(i,j)}t^{\vexc(\pi)}. 
\end{equation}
\end{conj}

Motivated by Gessel's combinatorial interpretation of a decomposition of formal Laurent series in terms of lattice paths~\cite{ges}, Wang and Zhao~\cite{zhao} found a {\em reversal-concatenation} decomposition of ballot permutations and proved Conjecture~\ref{conj:Spiro}. Afterwards, Sun and Zhao~\cite{SZ} completed a proof of Conjecture~\ref{conj:WZ}.  However, both proofs of Conjectures~\ref{conj:Spiro} and~\ref{conj:WZ} uses generating function and a bijective proof of~\eqref{Spiro}  still remains mysterious. On the other hand, the so-called {\em Foata's first fundamental transformation} (see~\cite[pp.~197-199]{lot}) on permutations establishes~\eqref{euler} bijectively. In this paper, we present an algorithm that decomposes ballot permutations into odd cycles, which proves a far reaching generalization of both Conjectures~\ref{conj:Spiro} and~\ref{conj:WZ}. 

Let $\pi\in\S_n$ and $k\in\{3,4,\ldots,n\}$. The value $k$ is called a  {\em peak} of $\pi$ if $1<\pi^{-1}(k)=\ell<n$ and $\pi(\ell-1)<k>\pi(\ell+1)$, and is called a {\em cyclic peak} of $\pi$ if $\pi^{-1}(k)<k>\pi(k)$. Deonte by $\Peak(\pi)$ (resp.~$\cPeak(\pi)$) the set of pairs $(b,\{a,c\})$ such that $b$ is a peak (resp.~cyclic peak) of $\pi$ and $a,c$ are neighbors (resp.~cyclic neighbors) of $b$ in $\pi$. For example, if $\pi = 839164752=(1, 8, 5, 6, 4)(2, 3, 9)(7)\in\S_9$, then 
\begin{align*}
\Peak(\pi)&=\{(9,\{1,3\}), (6,\{1,4\}),(7,\{4,5\})\},\\
\cPeak(\pi)&=\{(8,\{1,5\}), (6,\{4,5\}),(9,\{2,3\})\}.
\end{align*}

\begin{theorem}\label{thm:main}
There exists a bijection $\Psi: \B_n\rightarrow\Od_n$ that transforms the pair $(\des,\Peak)$ to $(\vexc,\cPeak)$. 
\end{theorem}

Our bijection $\Psi$ can be extended to {\em well-labelled positive paths}, a class of generalized ballot permutations arising from  polytope theory, that were introduced and enumerated by Bernardi,  Duplantier and Nadeau~\cite{BDN}.

\begin{table}
\hrule
{\small
\begin{tabbing}
xxxxxxxxxx\=xxxxxxxxxxxxxxxxxxxxxxxxxxxxxxxx\= xxxxxxxxxxxxxxxxxxxxxxxxxxxxxxx\=xxxxxxxxxxxxx\= \kill

Pattern $p$\> $\quad$First values of $|\B_n(p)|$: \> Notes\> OEIS seq.\\
\\
$123$\> $ 1, 1,2,2,5,5,14,14,32,32,\ldots$ \>  Catalan number $C(\lceil\frac{n}{2}\rceil)$  \> A208355  \\
$321$\> $1,1,3,9,28,90,297,1001, 3432,\ldots$ \> $\frac{3}{n+1}{2n-2\choose n-2}$ for $n>1$  \> A071724   \\
$132$\> $1,1,2,4,10,25,70,196,  588, 1764, \ldots$ \> Catalan product $C(\lfloor\frac{n}{2}\rfloor)C(\lfloor\frac{n+1}{2}\rfloor)$ \> A005817   \\
$231$\> Wilf-equivalent to pattern $132$ \>    \> A005817   \\
$213$\> $1,1,3,6,21,52,193,532, 2034,\ldots$\> Gessel walks ending on the $y$-axis  \> A151396    \\
$312$\>   Wilf-equivalent to pattern $213$ \>   \> A151396
\vspace{.1in}
\end{tabbing}
}
\hrule
\vspace{.2in}
\caption{Length-$3$ patterns  for Ballot permutations.\label{3-pattern}}
\end{table}

We will also investigate the enumerative aspect of pattern avoiding ballot permutations.
We say a word $w=w_1w_2\cdots w_n\in\Z^n$ {\em avoids the pattern} $\sigma=\sigma_1\sigma_2\cdots \sigma_k\in\S_k$ ($k\leq n$) if there does not exist $i_1<i_2<\cdots<i_k$ such that  the subword $w_{i_1}w_{i_2}\cdots w_{i_k}$ of $w$ is order isomorphism to $\sigma$. For a set $\mathcal{W}$ of words, let $\mathcal{W}(\sigma)$ be the set of $\sigma$-avoiding words in $\mathcal{W}$. Two patterns $\sigma$ and $\tau$ are said to be {\em Wilf-equivalent} over $\mathcal{W}$ if $|\mathcal{W}(\sigma)|=|\mathcal{W}(\tau)|$. 
One of the most famous enumerative results in pattern avoiding permutations, attributed to MacMahon and Knuth (see~\cite{kit,SS}), is that $|\S_n(\sigma)|=C(n)$ for each pattern $\sigma\in\S_3$, where 
$$C(n):=\frac{1}{n+1}\binom{2n}{n}
$$ is the $n$-th {\em Catalan number}. For pattern avoiding ballot permutations, our enumerative results are summarized in Table~\ref{3-pattern}, of which the most intriguing one is the connection between the pattern $213$ and Gessel walks.

A {\em Gessel walk} is a lattice path  confined to the quarter plane $\N^2$ using steps from the set $\{\uparrow,\downarrow,\nearrow,\swarrow\}$. See Fig.~\ref{G:GBwalks} for an example of Gessel walk. Around 2000, Ira Gessel conjectured that the number of $2n$-step Gessel walks that starting and ending at $(0, 0)$ has the simple hypergeometric formula 
$$
g_n=16^n\frac{(5/6)_n(1/2)_n}{(5/3)_n(2)_n},
$$
where $(a)_n:=a(a+1)\cdots(a+n-1)$ is the ascending factorial. The sequence $\{g_n\}_{n\geq0}$ appears as A135404 in the OEIS~\cite{oeis}. 
This attractive conjecture remained open for several years and now has at least three proofs: the original  proof was given by Kauers,  Koutschan and Zeilberger~\cite{KKZ} with the aid of computer, the second one was proposed by Bostan, Kurkova and Raschel~\cite{BKR} using deep machinery of complex analysis, and an elementary approach was finally found by Bousquet-M\'elou~\cite{BM2016} in $2016$. For the enumeration of other walks with small steps in the quarter plane, the reader is referred to~\cite{BMM}.

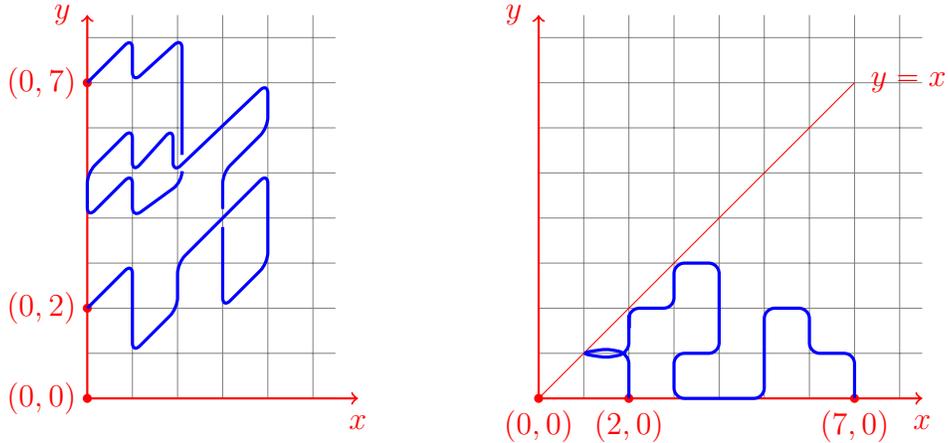
\begin{figure}
\begin{tikzpicture}[scale=0.6, axes/.style={red, thick}, walk/.style={blue, very thick, rounded corners}]
\draw[help lines] (0,0) grid (5.5, 8.5);
\begin{scope}[axes]
\draw[->] (0, 0) -- (6, 0) node[below=2pt] {$x$} coordinate(x axis);
\draw[->] (0, 0) -- (0, 8.5) node[left=2pt] {$y$} coordinate(y axis);
\foreach \y in {0,2,7}
{
\fill (0,\y) circle (0.1cm);
\draw[yshift=\y cm] node[left] {$(0,\y)$};
}
\end{scope}
\draw[walk]
(0,2) -- (1,3) -- (1,1) -- 
(2,2) -- (2,3) -- (4,5) -- 
(4,3) -- (3,2) -- (3, 3.8)
(3, 4.2) --(3,5) -- (4,6) -- 
(4,7) -- (1.9, 5) -- (1.9, 6) -- (1,5) -- 
(1,6) -- (0,5) -- 
(0,4) -- (1,5) -- 
(1,4) -- (2.1, 4.8) -- (2.1, 5)
(2.1, 5.4) -- (2.1, 8)  -- 
(1,7) -- (1,8) -- (0,7); 

\begin{scope}[xshift=10cm]
\draw[help lines] (0,0) grid (8.5, 8.5);
\begin{scope}[axes]
\draw[->] (0, 0) -- (8.5, 0) node[below=2pt] {$x$} coordinate(x axis);
\draw[->] (0, 0) -- (0, 8.5) node[left=2pt] {$y$} coordinate(y axis);
\foreach \x in {0,2,7} 
{
\fill (\x,0) circle (0.1cm);
\draw[xshift=\x cm] node[below] {$(\x,0)$};
}
\end{scope}
\draw[red]
(0,0) -- (7,7) node[right=2pt]{$y=x$};
\draw[walk] 
(2,0) -- (2, 1) .. controls (1.5, 1.1) .. (0.9, 1) .. controls (1.5, 0.9) .. (2, 1) -- 
(2,2) -- (2, 1.4) -- (2,2) -- 
(3,2) -- (3,3) -- (4,3) -- 
(4,1) -- (3,1) -- (3,0) -- 
(5,0) -- (5,2) -- (6,2) -- 
(6,1) -- (7,1) -- (7,0);
\end{scope}
\end{tikzpicture}
\caption{(Left) A $31$-step Gessel walk starting at $(0,2)$ and ending at $(0,7)$; (Right) a $19$-step Gouyou-Beauchamps walk starting at $(2,0)$ and ending at $(7,0)$.
\label{G:GBwalks}}
\end{figure}

For any $i,j\in\N$, let $\G(n;i,j)$ be the set of  $n$-step Gessel walks that starting at $(0,i)$ and ending at $(0, j)$. 
\begin{theorem}\label{thm:213}
For any $n\geq0$, 
$$
|\B_{n+1}(213)|=\sum_{j=0}^n|\G(n;0,j)|.
$$
\end{theorem} 
Along the way to the proof of Theorem~\ref{thm:213}, we have been able to find a certain  class of $213$-avoiding ballot permutations that are in bijection with $\G(n;i,j)$. Consequently,  a new combinatorial interpretation for the integer sequence $\{g_n\}_{n\geq0}$ is discovered. A ballot permutation $\pi\in\S_{2n+1}$ is called a {\em Dyck permutation} if $\asc(\pi)=\des(\pi)=n$. 
Bidkhori and Sullivant~\cite{BS} showed that Dyck permutations are enumerated by the {\em Eulerian-Catalan numbers}.
A specialization of our result indicates that $213$-avoiding Dyck permutations of length $2n+1$ are counted by $g_n$. 

The rest of this paper is organized as follows. In Section~\ref{sec:2}, we introduce a technical bijection that forms the main ingredient in our construction of $\Psi$. Based on this bijection, we construct $\Psi$ and extend it to well-labelled positive paths in Section~\ref{sec:3}. Section~\ref{sec:4} is devoted to the enumeration of ballot permutations avoiding patterns of length $3$. Finally, in Section~\ref{sec:5} we conclude our paper with several problems for further research.

\section{A technical bijection}
\label{sec:2}

Our decomposition of ballot permutations is based on a crucial bijection that we introduce in this section. 

Let us begin with the introduction of two useful combinatorial objects. Let $A=\{a_1,a_2,\ldots,a_k\}$ be a set of positive integers such that $a_1<a_2<\cdots <a_k$. A word $w$ over the alphabet $A\cup\{\Box\}$ is called a {\em box-permutation} on $A$ if each letter in $A$ appears exactly once in $w$ and  
\begin{itemize}
\item if  $\Box$ does not appear in $w$, then $w=a_1a_2\cdots a_k$;
\item otherwise, the subword of $w$ formed by integers before the first (or after the last) box  has length at least one and is increasing, and the subword of $w$ formed by integers between two successive boxes has length at least two and is increasing. 
\end{itemize}
Roughly speaking, box-permutations on $A$ are some permutations of $A$ that are separated properly  into several increasing segments by boxes.  Denote by $\BP(A)$ the set of all box-permutations on $A$. For example, if $A=[3]$, then
$$
\BP(A)=\{123,1\Box23,2\Box13,3\Box12,12\Box3,13\Box2,23\Box1\}. 
$$
A cycle $C=(c_1c_2\cdots c_{2\ell+1})$ of length $2\ell+1$, $\ell\geq0$, with letters from $A\cup\{\Box\}$ is called an {\em odd box-cycle} (resp.~a {\em reverse odd box-cycle}) on $A$ if each letter in $A$ appears in $C$ at most once and 
\begin{itemize}
\item either $C=(a)$ for a single element $a\in A$,
\item or $\ell\geq1$ and $C$ contains at least one box and each maximal component of $C$ formed by only integers has length at least two and is increasing clockwise (resp.~counterclockwise).
\end{itemize}
For example, the cycle $(1356\Box279\Box)$ is an odd box-cycle, while $(96531\Box 72\Box)$ is a reverse odd box-cycle. Note that as an odd box-cycle, $(1356\Box279\Box)=(56\Box279\Box13)$, and its two maximal components  formed by only integers are $1356$ and $279$ (written in clockwise order). A group of cycles   is called a {\em cyclic box-permutation} on $A$ if
\begin{itemize} 
\item each cycle is either  an odd box-cycle or a reverse odd box-cycle on $A$,
\item and each element of $A$ appears once in exactly  one of the cycles. 
\end{itemize}
For example, the cycles $(16\Box279\Box)(53\Box)(4)(8)$ is a cyclic box-permutation on $[9]$. 
Denote by $\cBP(A)$ the set of all cyclic box-permutations on $A$. For example, if $A=[3]$, then
$$
\cBP(A)=\{(1)(2)(3),(1)(23\Box),(1)(32\Box),(2)(13\Box),(2)(31\Box),(3)(12\Box),(3)(21\Box)\}. 
$$
For a (cyclic) box-permutation $\pi$, the set of all the $2$-subsets $\{a,b\}$ such that $a$ and $b$ are the two neighbors of a certain box in $\pi$  is called the {\em box-neighbor-set} of $\pi$. For instance, the box-neighbor-set of the box-permutation $23\Box4679\Box18\Box5$ is $\{\{3,4\},\{1,9\},\{5,8\}\}$, while  the box-neighbor-set of $(16\Box279\Box)(53\Box)(4)(8)$ is $\{\{2,6\},\{1,9\},\{3,5\}\}$.

The reason to use box as a letter will become transparent in next section, where each box in words receives a Dyck permutation. 
The following bijection between $\BP(A)$ and $\cBP(A)$ forms the main ingredient for our construction of the aforementioned map $\Psi$. 
\begin{theorem}\label{thm:box}
For any finite set $A=\{a_1<a_2<\cdots<a_k\}$ of positive integers,
there exists a box-neighbor-set-preserving bijection $\psi$ from $\BP(A)$ to $\cBP(A)$.
\end{theorem}

A crucial lemma is needed before we proceed to prove Theorem~\ref{thm:box}. For any $A=\{a_1<a_2<\cdots<a_k\}$, let us divide $\BP(A)$ into three disjoint subsets as follows:
\begin{itemize}
\item $\BP^{(1)}(A)$ is the set of box-permutations $\pi\in\BP(A)$ such that either $a_1$ appears in an even position of $\pi$, or $\pi_1=a_1$ and $\pi_2=\Box$;
\item $\BP^{(2)}(A):=\{\pi\in\BP(A): \pi_1=a_1\text{ and  $\pi_2\neq\Box$}\}$; 
\item $\BP^{(3)}(A):=\BP(A)\setminus(\BP^{(1)}(A)\cup\BP^{(2)}(A))$, consisting of box-permutations $\pi\in\BP(A)$ such that  $a_1$ appears in an odd position, other than the first position, of $\pi$.
\end{itemize}
For example, if $A=[3]$, then 
$$
\BP^{(1)}(A)=\{1\Box23,23\Box1\}, \BP^{(2)}(A)=\{123,12\Box3,13\Box2\} \text{ and }\BP^{(3)}(A)=\{2\Box13,3\Box12\}.
$$

\begin{lemma}
There exists a box-neighbor-set-preserving bijection $\phi$ between $\BP^{(1)}(A)$ and $\BP^{(3)}(A)$.
\end{lemma}
\begin{proof}
For any box-permutation $\pi\in\BP(A)$, denote by $\mathsf{F}(\pi)=\pi_1$ the first letter of $\pi$, by $\mathsf{P}(\pi)$ the position of $a_1$ and by $\mathsf{N}(\pi)$ the right neighbor (if any) of $a_1$  in $\pi$. These three indices are useful in considering the inverse of $\phi$. 
Now suppose $\pi=\pi_1\pi_2\cdots\pi_n\in\BP^{(1)}(A)$ and we aim to construct $\phi(\pi)\in\BP^{(3)}(A)$ according to the following four cases.

{\bf Case I: $\pi_1=a_1$ and $\pi_2=\Box$}. In this case, 
$$\pi=\blue{a_1\Box\pi_3}\pi_4\cdots\pi_n \text{ and define } \phi(\pi)=\blue{\pi_3\Box a_1}\pi_4\pi_5\cdots\pi_n,
$$ the permutation obtained from $\pi$ by switching the first and the third letters. Now for $\phi(\pi)=\pi'$, we have $\mathsf{P}(\pi')=3$, $\pi'_2=\Box$ and either $n=3$ or $\mathsf{F}(\pi')<\mathsf{N}(\pi')$. 

{\bf Case II: $\pi_1\neq a_1$}. In this case, $a_1$ must be the right neighbor of a box in $\pi$.  Let $x$ be the left neighbor of the first box, if any, to the right of $a_1$. Otherwise, no box appears to the right of $a_1$ and we set $x=+\infty$.  Let $a$ be the left neighbor of the first box in $\pi$. Suppose that $\pi_i=a$, $\pi_j=a_1$ and $\pi_k=x$ (whenever $x=+\infty$, set $k=n+1$) for some $i<j<k$. 
Let us consider the set of letters 
$$
T:=\{\pi_{\ell}: 1\leq \ell<i\text{ or }j<\ell<k\}
$$
and further   distinguish three subcases.
\begin{enumerate}
\item $T\neq\emptyset$ and $t=\min(T)<\min\{a,x\}$. In this case, either $t$ is the first letter of $\pi$ or $t$ is the letter immediately after $a_1$. We can flip $t$ back and forth between ``the first letter of $\pi$'' and ``the letter immediately after $a_1$'' to get $\phi(\pi)$. To be more precise, if $t=\pi_1$ is the first letter of $\pi$, then 
$$
\pi=\blue{t}\pi_2\cdots a\Box\cdots\Box a_1\cdots
$$
and define 
$$
\phi(\pi)=\pi_2\cdots a\Box\cdots\Box a_1\blue{t}\cdots.
$$
Otherwise, $t=\pi_{j+1}$ and $\pi$ can be written as 
$$
\pi=\pi_1\cdots a\Box\cdots\Box a_1\blue{t}\cdots,
$$
then we define 
$$
\phi(\pi)=\blue{t}\pi_1\cdots a\Box\cdots\Box a_1\cdots.
$$
For  the former $\phi(\pi)=\pi'$, we have $\mathsf{P}(\pi')\geq3$, $\mathsf{F}(\pi')>\mathsf{N}(\pi')$ and $\mathsf{N}(\pi')$ is not the left neighbor of a box, while for the later $\phi(\pi)=\pi'$, we have $\mathsf{P}(\pi')\geq5$, $\pi'_2\neq\Box$ and either $n=\mathsf{P}(\pi')$ or $\mathsf{F}(\pi')<\mathsf{N}(\pi')$. 

\item $x<\pi_1<\cdots<\pi_i=a$ and $k=j+1$. In this case, $\pi$ has the form
$$
\pi=\pi_1\cdots a\Box\cdots \Box a_1\blue{x\Box\pi_{k+2}}\cdots,
$$
then define 
$$
\phi(\pi)=\blue{\pi_{k+2}\Box x}\pi_1\cdots a\Box\cdots \Box a_1\cdots, 
$$
i.e., delete the subword $x\Box\pi_{k+2}$ from $\pi$  and insert $\pi_{k+2}\Box x$ at the beginning. Now for $\phi(\pi)=\pi'$, we have $\mathsf{P}(\pi')\geq7$, $\pi'_2=\Box$ and either $\mathsf{P}(\pi')=n$ or $\mathsf{F}(\pi')<\mathsf{N}(\pi')$. 
\item $\pi_1=a$ and $a<\pi_{j+1}<\cdots <\pi_k=x$ (possibly $j=n$). In this case, $\pi$ has the form
$$
\pi=\blue{a\Box\pi_3}\pi_4\cdots \Box a_1\pi_{j+1}\cdots x\Box\cdots,
$$
then define 
$$
\phi(\pi)=\pi_4\cdots \Box a_1\blue{\pi_3\Box a}\pi_{j+1}\cdots x\Box\cdots,
$$
i.e., delete the subword $a\Box\pi_3$ from $\pi$ and then insert $\pi_3\Box a$ immediately after $a_1$. Now for $\phi(\pi)=\pi'$, we have $\mathsf{P}(\pi')\geq3$, $\mathsf{F}(\pi')>\mathsf{N}(\pi')$ and $\mathsf{N}(\pi')$ is the left neighbor of a box. 
\end{enumerate}

It is clear from the above construction that the map $\phi:\BP^{(1)}(A)\rightarrow\BP^{(3)}(A)$ preserves the box-neighbor-set statistic. We observe that each of the four cases in the construction of  $\phi$ is reversible. Given any 
$\pi'=\pi'_1\pi'_2\cdots\pi'_n\in\BP^{(3)}(A)$, we consider  the following four cases:
\begin{itemize}
\item If $\mathsf{P}(\pi')=3$, $\pi'_2=\Box$ and either $n=3$ or $\mathsf{F}(\pi')<\mathsf{N}(\pi')$, then we can apply the reverse operation in {\bf Case~I} to get $\phi^{-1}(\pi')$. 
\item If $\mathsf{P}(\pi')\geq5$, $\pi'_2\neq\Box$ and either $\mathsf{P}(\pi')=n$ or $\mathsf{F}(\pi')<\mathsf{N}(\pi')$, then we can apply the reverse operation in the second situation of {\bf Case~II}(1) to get $\phi^{-1}(\pi')$.
\item If $\mathsf{P}(\pi')\geq7$, $\pi'_2=\Box$ and either $\mathsf{P}(\pi')=n$ or $\mathsf{F}(\pi')<\mathsf{N}(\pi')$, then we can apply the reverse operation in {\bf Case~II}(2) to get $\phi^{-1}(\pi')$.
\item If $\mathsf{P}(\pi')\geq3$, $\mathsf{F}(\pi')>\mathsf{N}(\pi')$ and $\mathsf{N}(\pi')$ is not the left neighbor of a box, then we can apply the reverse operation in the first situation of {\bf Case~II}(1) to get $\phi^{-1}(\pi')$.
\item If $\mathsf{P}(\pi')\geq3$, $\mathsf{F}(\pi')>\mathsf{N}(\pi')$ and $\mathsf{N}(\pi')$ is the left neighbor of any box,  then we can apply the reverse operation in {\bf Case~II}(3) to get $\phi^{-1}(\pi')$.
\end{itemize}
Since the above five cases are disjoint and exhaust all possible $\pi'\in\BP^{(3)}(A)$, the map $\phi$ is a bijection.
\end{proof}

Now we are in a position to prove Theorem~\ref{thm:box}. 

\begin{proof}[{\bf Proof of Theorem~\ref{thm:box}}] For any box-permutation $\pi=\pi_1\pi_2\cdots\pi_n\in\BP(A)$, we construct $\psi(\pi)\in\cBP(A)$ by factoring out odd cycles step by step according to the following three cases:
\begin{enumerate}
\item If $\pi\in\BP^{(2)}(A)$, then form the cycle $C_1=(a_1)$ and the box-permutation $\tilde\pi=\pi_2\pi_3\cdots\pi_n\in\BP(A\setminus\{a_1\})$. Define 
$$
\psi(\pi)=\{C_1,\psi(\tilde\pi)\}.
$$
\item  If $\pi\in\BP^{(3)}(A)$, then the letter $a_1$ appears in position $2\ell+1$ of $\pi$ for some $\ell\geq1$. Form the \blue{odd box-cycle} $C_1=(a_1\pi_{1}\pi_2\cdots\pi_{2\ell})$ and the box-permutation $\tilde\pi=\pi_{2\ell+2}\pi_{2\ell+3}\cdots\pi_n\in\BP(\tilde A)$, where $\tilde{A}=A\setminus\{\pi_i:1\leq i\leq {2\ell+1}\text{ and } \pi_i\neq\Box\}$. Define 
$$
\psi(\pi)=\{C_1,\psi(\tilde\pi)\}.
$$
\item If $\pi\in\BP^{(1)}(A)$, then $\pi'=\phi(\pi)\in \BP^{(3)}(A)$. Now the letter $a_1$ appears in position $2\ell+1$ of $\pi'$ for some $\ell\geq1$. Form the \blue{reverse odd box-cycle} $C_1=(a_1\pi'_{2\ell}\pi'_{2\ell-1}\cdots\pi'_{1})$ and the box-permutation $\tilde\pi'=\pi'_{2\ell+2}\pi'_{2\ell+3}\cdots\pi'_n\in\BP(\tilde A)$, where $\tilde{A}=A\setminus\{\pi'_i:1\leq i\leq {2\ell+1}\text{ and } \pi'_i\neq\Box\}$. Define 
$$
\psi(\pi)=\{C_1,\psi(\tilde\pi')\}.
$$
\end{enumerate}

It is clear from the above construction that $\psi: \BP(A)\rightarrow\cBP(A)$ is a bijection as $\phi:\BP^{(1)}(A)\rightarrow\BP^{(3)}(A)$ does. 
\end{proof}

\begin{example}Let $A=\{2,4,5,6,7,8,9,11,14,16,17,18\}$. Consider the box-permutation
$$
\pi=\blue{4}\,6\,\Box\,2\,7\,16\,\Box\,9\,11\,14\,\Box\,5\,18\,\Box\,8\,17\in\BP(A). 
$$
Applying the operation in {\bf Case II}(1) to $\pi$ we get 
$$\phi(\pi)=\blue{6\,\Box\,2}\,4\,7\,16\,\Box\,9\,11\,14\,\Box\,5\,18\,\Box\,8\,17$$
and so 
\begin{equation}\label{exm:1}
\psi(\pi)=\{(6,2,\Box),\psi(\blue{4}\,7\,16\,\Box\,9\,11\,14\,\Box\,5\,18\,\Box\,8\,17)\}=\{(6,2,\Box),(4),\psi(\pi')\},
\end{equation}
where $\pi'=\blue{7}\,16\,\Box\,9\,11\,14\,\Box\,5\,18\,\Box\,8\,17$. 
Applying the operation in {\bf Case II}(1) to $\pi'$  we get
$
\phi(\pi')=\blue{16\,\Box\,9\,11\,14\,\Box\,5}\,7\,18\,\Box\,8\,17
$
and so 
\begin{align*}
\psi(\pi')&=\{(16,5,\Box,14,11,9,\Box),\psi(7\,18\,\Box\,8\,17)\}\\
&=\{(16,5,\Box,14,11,9,\Box),(7),(8,18,\Box),(17)\}.
\end{align*}
Thus,  by~\eqref{exm:1} we have 
\begin{equation}\label{exm:box}
\psi(\pi)=\{(6,2,\Box),(4),(16,5,\Box,14,11,9,\Box),(7),(8,18,\Box),(17)\}.
\end{equation}

Consider another box-permutation
$$
\pi=\blue{2\,\Box\,8}\,9\,\Box\,7\,14\,17\,18\,\Box\,4\,11\,\Box\,5\,6\,\Box\,16\in\BP(A). 
$$
Applying the operation in {\bf Case I} to $\pi$ we get 
$$\phi(\pi)=\blue{8\,\Box\,2}\,9\,\Box\,7\,14\,17\,18\,\Box\,4\,11\,\Box\,5\,6\,\Box\,16$$
and so 
\begin{equation}\label{exm:2}
\psi(\pi)=\{(8,2,\Box),\psi(\pi')\},
\end{equation}
where $\pi'=\blue{9\,\Box\,7}\,14\,17\,18\,\Box\,4\,11\,\Box\,5\,6\,\Box\,16$. Applying the operation in {\bf Case II}(3) to $\pi'$ we get 
$$
\phi(\pi')=14\,17\,18\,\Box\,4\,7\,\Box\,9\,11\,\Box\,5\,6\,\Box\,16
$$
and so $\psi(\pi')=\{(18,17,14,4,\Box),\psi(\pi'')\}$, where $\pi''=7\,\Box\,9\,11\,\Box\,5\,6\,\Box\,16$. Applying the operation in {\bf Case II}(2) to $\pi''$ we get 
$$
\phi(\pi'')=16\,\Box\,6\,7\,\Box\,9\,11\,\Box\,5
$$ and so $\psi(\pi'')=(16,5,\Box,11,9,\Box,7,6,\Box)$. Therefore, by~\eqref{exm:2} we have  
\begin{equation}\label{exm2:box}
\psi(\pi)=\{(8,2,\Box),(18,17,14,4,\Box),(16,5,\Box,11,9,\Box,7,6,\Box)\}.
\end{equation}
\end{example}

\section{ The construction of $\Psi$ and extension}%
\label{sec:3}

This section is devoted to the construction of our main bijection $\Psi$ stated in Theorem~\ref{thm:main}.

\begin{figure}
\begin{tikzpicture}[scale=0.55, axes/.style={black}, thick, shape=circle, inner sep=0.5mm]
\draw[help lines] (0,0) grid (20.5, 8.5);
\path
(0,0)  node[draw=blue, label=left: \textcolor{blue}{4}](4){}
(1,1)  node[draw=blue, label=above:\textcolor{blue}{6}](6){}
(7,1)  node[draw=blue, label=above:\textcolor{blue}{2}](2){}
(8,2)  node[draw=blue, label=above:\textcolor{blue}{7}](7){}
(9,3)  node[draw=blue, label=above:\textcolor{blue}{16}](16){}
(11,3) node[draw=blue, label=above:\textcolor{blue}{9}](9){}
(12,4) node[draw=blue, label=above:\textcolor{blue}{11}](11){}
(13,5) node[draw=blue, label=above:\textcolor{blue}{14}](14){}
(15,5) node[draw=blue, label=above:\textcolor{blue}{5}](5){}
(16,6) node[draw=blue, label=above:\textcolor{blue}{18}](18){}
(18,6) node[draw=blue, label=above:\textcolor{blue}{8}](8){}
(19,7) node[draw=blue, label=above:\textcolor{blue}{17}](17){};
\path
(2,2)  node[draw=red, label=above:\textcolor{red}{10}](10){}
(3,3)  node[draw=red, label=above:\textcolor{red}{13}](13){}
(4,2)  node[draw=red, label=above:\textcolor{red}{12}](12){}
(5,1)  node[draw=red, label=above:\textcolor{red}{1}](1){}
(6,2)  node[draw=red, label=above:\textcolor{red}{3}](3){}
(10,4) node[draw=red, label=above:\textcolor{red}{20}](20){}
(14,6) node[draw=red, label=above:\textcolor{red}{15}](15){}
(17,7) node[draw=red, label=above:\textcolor{red}{19}](19){};
\draw[very thick, blue] 
(4) -- (6)
(2) -- (7) -- (16)
(9) -- (11) -- (14)
(5) -- (18)
(8) -- (17);
\draw[very thick, red]  
(6) -- (10) -- (13) -- (12) -- (1) -- (3) -- (2)
(16) -- (20) -- (9)
(14) -- (15) -- (5)
(18) -- (19) -- (8);
\begin{scope}[axes]
\draw[->] (4) -- (21, 0) node[below=2pt] {$x$} coordinate(x axis);
\draw[->] (4) -- (0, 9)  node[left=2pt]  {$y$} coordinate(y axis);
\end{scope}
\end{tikzpicture}
\caption{The ballot path  associated with the ballot permutation  $\pi=4\,6\,10\,13\,12\,1\,3\,2\,7\,16\,20\,9\,11\,14\,15\,5\,18\,19\,8\,17\in\BP_{20}$.
\label{ballot:paths}}
\end{figure}
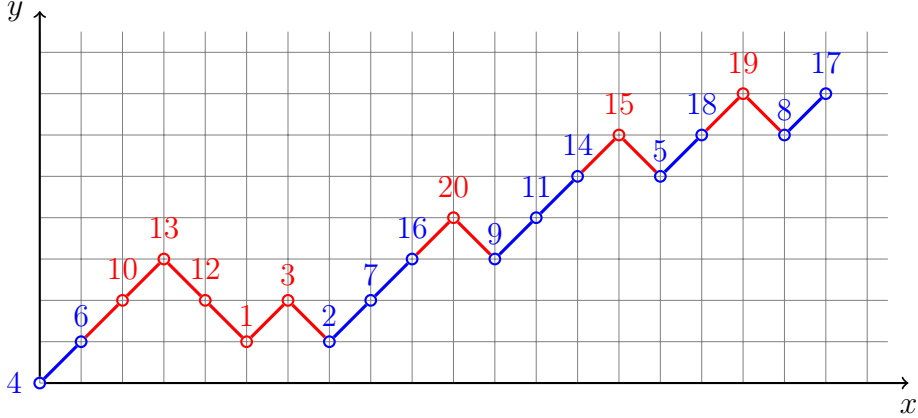

A lattice path confined to the quarter plane $\N^2$ using steps from the set $\{\nearrow,\searrow\}$ is called a {\em ballot path}. A {\em Dyck path} is a ballot path that begins with $(0,0)$ and ends at the $x$-axis. Every ballot permutation $\pi\in\B_n$ can be associated with a ballot path $L(\pi)$ starting at $(0,0)$ whose $i$th step is  $U=(1,1)$ or $D=(1,-1)$ according to 
$\pi_i<\pi_{i+1}$ or $\pi_i>\pi_{i+1}$. The $i$th lattice point of such a path will be labeled by the $i$th letter $\pi_i$ of $\pi$. See Fig.~\ref{ballot:paths} for an example of a labeled ballot path associated with a ballot permutation.  The letter $\pi_i$ (or the $i$th lattice point of $L(\pi)$) is said to {\em has height $h$} if
$$
\asc(\pi_1\pi_2\cdots\pi_i)-\des(\pi_1\pi_2\cdots\pi_i)=h.
$$
Clearly, a ballot permutation is a Dyck permutation if its associated ballot path is a Dyck path, i.e., the height of the last letter has height zero. A letter $\pi_i$ is called an {\em ascent bottom} (resp.~ a {\em descent bottom}) of $\pi$ if $\pi_i<\pi_{i+1}$ (resp.~$\pi_{i-1}>\pi_i$). 

{\bf The algorithm $\Psi$.} Given a permutation $\pi\in\B_n$, we perform the following three steps to get $\Psi(\pi)\in\Od_n$:
\begin{enumerate}
\item Firstly, we  decompose $\pi$ into a pair $(D,\sigma)$, where $D$ is the set of all maximal consecutive subwords of $\pi$ that are Dyck permutations of size at least $3$ and $\sigma$ is the box-permutation obtained from $\pi$ by replacing each maximal Dyck permutation (retaining its two boundary letters)  in $D$ by a box. We could find out all the maximal Dyck permutations (of size at least $3$) like this:
\begin{itemize}
\item[(i)] find the right-most descent bottom with the minimum height in $\pi$, say $\pi_i$ whose height is $h_1$; find the left-most ascent bottom with height $h_1$, say $\pi_{i'}$ with $i'<i-1$; then the interval $\pi_{i'}\pi_{i'+1}\cdots\pi_i$ forms the first maximal Dyck permutation;
\item[(ii)] if the suffix $\bar\pi=\pi_{i+1}\pi_{i+2}\cdots\pi_n$ is increasing, then stop finding; otherwise, continue to extract the next maximal Dyck permutation as (i) with $\pi$ replaced by  $\bar\pi$. 
\end{itemize}

 For the example of  ballot permutation in Fig.~\ref{ballot:paths}, $D$ contains four Dyck permutations which are painting in red and $\sigma=\blue{4\,6}\,\Box\,\blue{2\,7\,16}\,\Box\,\blue{9\,11\,14}\,\Box\,\blue{5\,18}\,\Box\blue{\,8\,17}$ is a box-permutation. 
\item Next, using the map $\psi$ in Theorem~\ref{thm:box} to form the cyclic box-permutation $\psi(\sigma)$. 
Continuing with the example, we get the cyclic box-permutation
$$
\psi(\sigma)=(6,2,\Box)(4)(16,5,\Box,14,11,9,\Box)(7)(8,18,\Box)(17)
$$
 from~\eqref{exm:box}. 
 \item Finally, inserting each Dyck permutation in $D$ back to its corresponding box in $\psi(\sigma)$ {\em properly} to form  $\Psi(\pi)$. Here proper means that for each Dyck permutation $d=d_1d_2\cdots d_k$ from $D$, 
 \begin{itemize}
 \item if the odd box-cycle  in $\psi(\sigma)$ containing the letters $d_1$ and $d_k$ has the form $(\ldots,d_1,\Box,d_k,\ldots)$, then the resulting odd cycle is $(\ldots,d_1,d_2,\ldots,d_k,\ldots)$;
 \item otherwise, the odd box-cycle  in $\psi(\sigma)$ containing the letters $d_1$ and $d_k$ has the form $(\ldots,d_k,\Box,d_1,\ldots)$, and the resulting odd cycle is $(\ldots,d_k,d_{k-1},\ldots,d_1,\ldots)$.
 \end{itemize}
 Since $\psi$ is proved in  Theorem~\ref{thm:box} to be box-neighbor-set-preserving, such inserting operation is possible. 
 
 Continuing with the above example, after inserting all the four Dyck permutations back to $\psi(\sigma)$, we get the odd order permutation
 $$
 \Psi(\pi)=(6,2,\red{3,1,12,13,10})(4)(16,5,\red{15},14,11,9,\red{20})(7)(8,18,\red{19})(17). 
 $$
 \end{enumerate}

 It is clear from the above construction that $\Psi$ is box-neighbor-set-preserving. To see that $\Psi$ is bijective, we aim to construct its inverse $\Psi^{-1}$ explicitly. 
 
 {\bf The algorithm $\Psi^{-1}$.} Suppose that  $\pi\in\Od_n$ and $\pi=C_1C_2\cdots C_{\ell}$, where each $C_i$ is an odd cycle. We perform the following three steps to get $\Psi^{-1}(\pi)\in\B_n$:
 \begin{enumerate}
 \item[(I)] Firstly, we  decompose $\pi$ into a pair $(D,\sigma)$, where 
\begin{itemize}
\item $D$ is the set of all maximal consecutive subwords of all cycles of $\pi$ that are Dyck permutations of size at least $3$;
\item and $\sigma$ is the cyclic box-permutation formed by all cyclic box-permutations obtained from  all odd cycles of $\pi$ by replacing each maximal Dyck permutation (retaining its two boundary letters)  in $D$ by a box.
\end{itemize}  
For each odd cycle $C=(c_1c_2\cdots c_k)$ of $\pi$ with length $k\geq3$ and $c_1$ is any cyclic peak of $C$ (i.e.~$c_k<c_1>c_2$),  we could find out the maximal Dyck permutation that contains $c_1$ in $C$ like this: 
\begin{itemize}
\item For each $2\leq i\leq k$, define 
$$h_i^+=\asc(c_1c_2\cdots c_i)-\des(c_1c_2\cdots c_i)$$
and
$$h_i^-=\des(c_1c_{k}c_{k-1}\cdots c_{i})-\asc(c_1c_{k}c_{k-1}\cdots c_{i}).$$ 
\item We need to distinguish two cases:
\begin{enumerate}
\item If $h_k^+\geq0$, then find the greatest index $r$ such that 
$$h_r^+=\min\{h_{i}^+: 2\leq i\leq k-1\}$$ and find the smallest index $l$, $l>r$, such that 
$$h_l^-=h_r^+\quad\text{and} \quad h_l^-=\min\{h_{i}^-: l\leq i\leq k\}.$$
\item Otherwise, $h_k^+<0$. Find the smallest index $l$ such that 
$$h_l^-=\min\{h_{i}^-: 3\leq i\leq k\}$$
 and then find the greatest index $r$, $r<l$, such that 
$$h_r^+=h_l^-\quad\text{and} \quad h_r^+=\min\{h_{i}^+: 2\leq i\leq r\}.$$
\end{enumerate}
\item The consecutive subword $c_lc_{l+1}\cdots c_kc_1 c_2\cdots c_r$ of $C$ is the maximal Dyck permutation in $C$  that contains $c_1$.
\end{itemize}
Note that no two maximal Dyck permutations of a cycle $C$ can overlap. For example,  the interior of each maximal Dyck permutation in
 $$
\pi=(6,2,\red{3,1,12,13,10})(4)(16,5,\red{15},14,11,9,\red{20})(7)(8,18,\red{19})(17). 
 $$
 are colored red. Replacing the interior of each maximal Dyck permutation in $\pi$ by a box gives the cyclic box-permutation
 $$
 \sigma=(6,2,\Box)(4)(16,5,\Box,14,11,9,\Box)(7)(8,18,\Box)(17).
 $$

\item[(II)] Next, using the map $\psi^{-1}$ in Theorem~\ref{thm:box} to form the  box-permutation $\psi^{-1}(\sigma)$. 
Continuing with the example, we get the box-permutation
$$
\psi^{-1}(\sigma)=\blue{4\,6}\,\Box\,\blue{2\,7\,16}\,\Box\,\blue{9\,11\,14}\,\Box\,\blue{5\,18}\,\Box\blue{\,8\,17}.
$$
 \item[(III)] Finally, inserting each Dyck permutation in $D$ back to its corresponding box in $\psi^{-1}(\sigma)$ {\em properly} (in the sense of step (3) of the algorithm $\Psi$) to form  $\Psi^{-1}(\pi)$.
 \end{enumerate}

Since the map $\psi$ in  Theorem~\ref{thm:box} is bijective, it is routine to check that $\Psi$ and $\Psi^{-1}$ are inverse to each other, which completes the proof of Theorem~\ref{thm:main}.

 \subsection{Extend $\Psi$ to well-labelled positive paths}
 In the rest of this section, we aim to extend $\Psi$ from ballot permutations to well-labelled positive paths. As an application, the generating function for the number of  well-labelled positive paths is calculated smoothly. 
 
 A lattice path confined to the quarter plane $\N^2$ using steps from the set $\{\nearrow,\rightarrow,\searrow\}$ is called a {\em positive  path}. A {\em well-labelled positive path} of size $n$ is a pair $(p,\pi)$, where  
 \begin{itemize}
 \item $p$ is a positive path with $n-1$ steps starting at $(0,0)$ and
 \item  $\pi$ is a permutation in $\S_n$ satisfying  $\pi_i<\pi_{i+1}$ (resp.~$\pi_i>\pi_{i+1}$) whenever the $i$th step of $p$ is $\nearrow$ (resp.~$\searrow$) for each $i\in[n-1]$. 
 \end{itemize}
 See Fig.~\ref{positive:paths} for  a well-labelled positive path of size $20$.  Let $\PP_n$ be the set of all well-labelled positive paths of size $n$. 
 
\begin{figure}
\begin{tikzpicture}[scale=0.55, axes/.style={black}, thick, shape=circle, inner sep=0.5mm]
\draw[help lines] (0,0) grid (20.5, 7.5);
\path
(0,0)  node[draw=blue, label=below:\textcolor{blue}{2}](2){}
(5,0)  node[draw=blue, label=below:\textcolor{blue}{8}](8){}
(6,1)  node[draw=blue, label=above:\textcolor{blue}{9}](9){}
(8,1)  node[draw=blue, label=above:\textcolor{blue}{7}](7){}
(9,2)  node[draw=blue, label=above:\textcolor{blue}{14}](14){}
(10,3) node[draw=blue, label=above:\textcolor{blue}{17}](17){}
(11,4) node[draw=blue, label=above:\textcolor{blue}{18}](18){}
(12,4) node[draw=blue, label={[above=2pt]:\textcolor{blue}{4}}](4){}
(13,5) node[draw=blue, label=above:\textcolor{blue}{11}](11){}
(17,5) node[draw=blue, label=above:\textcolor{blue}{5}](5){}
(18,6) node[draw=blue, label={[above=2pt]:\textcolor{blue}{6}}](6){}
(19,6) node[draw=blue, label=above:\textcolor{blue}{16}](16){};
\path
(1,1)  node[draw=red, label=above:\textcolor{red}{10}](10){}
(2,1)  node[draw=red, label=above:\textcolor{red}{15}](15){}
(3,2)  node[draw=red, label=above:\textcolor{red}{19}](19){}
(4,1)  node[draw=red, label=above:\textcolor{red}{12}](12){}
(7,1)  node[draw=red, label=above:\textcolor{red}{1}](1){}
(14,5) node[draw=red, label=above:\textcolor{red}{13}](13){}
(15,6) node[draw=red, label=above:\textcolor{red}{20}](20){}
(16,5) node[draw=red, label=above:\textcolor{red}{3}](3){};
\draw[very thick, blue] 
(8) -- (9)
(7) -- (14) -- (17) -- (18) 
(4) -- (11)
(5) -- (6);
\draw[very thick, red]  
(2) -- (10) -- (15) -- (19) -- (12) -- (8)
(9) -- (1) -- (7)
(18) -- (4)
(11) -- (13) -- (20) -- (3) -- (5)
(6) -- (16);
\begin{scope}[axes]
\draw[->] (2) -- (21, 0) node[below=2pt] {$x$} coordinate(x axis);
\draw[->] (2) -- (0, 8)  node[left=2pt]  {$y$} coordinate(y axis);
\end{scope}
\end{tikzpicture}
\caption{
A well-labelled positive path of size $20$.
\label{positive:paths}}
\end{figure}
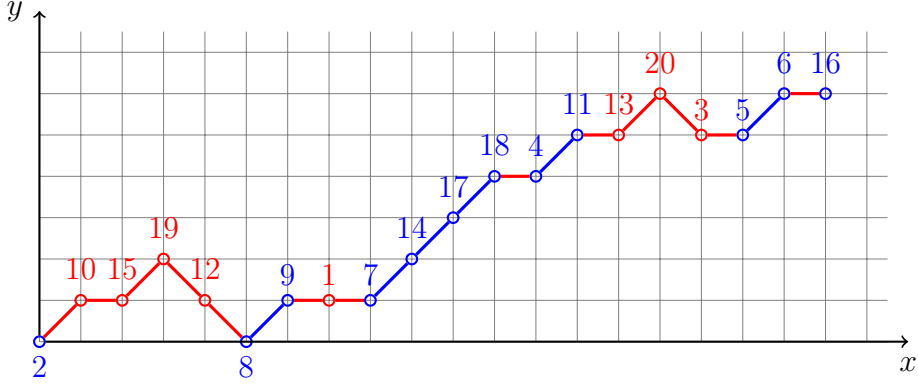

In order to extend $\Psi$ to well-labelled positive paths, we need to introduce a generalization of odd cycles. A word with distinct letters from $\Po$ is called a {\em cluster}. For convenience, we underline a cluster when it has at least two letters, i.e., size at least two. 
Two clusters are disjoint if they contain no common letter.  
A cycle $C=(c_1c_2\cdots c_k)$ is called a {\em  cluster-cycle} of order $k$ if all letters $c_i$'s are disjoint clusters.  Note that any cyclic shifting $(c_ic_{i+1}\cdots c_kc_1\cdots c_{i-1})$ is considered as the same cluster-cycle as $C$. A cluster-cycle of odd order is called an {\em odd cluster-cycle}. For example, $(2,\underline{9\,1\,7},\underline{11\,13},3)$ is a cluster-cycle of order $4$. 
A group of odd cluster-cycles   is called an {\em odd order cluster-permutation} (OCP for short) of size  $n$ if
\begin{itemize} 
\item each cluster of the cycles is a word over $[n]$,
\item and each element of $[n]$ appears once in exactly  one of the cluster of the cycles. 
\end{itemize}
For example, $(\underline{13},\underline{47},6)(\underline{25})$ is an OCP of size $7$. Note that odd order permutations are OCPs without any underlined cluster. Denote by $\C_n$ the set of all OCPs of size $n$. 

For the sake of convenience, we will represent a well-labelled positive path $(p,\pi)\in\PP_n$ by a permutation $\underline{\pi}$ whose letters are clusters, that we call a {\em cluster-permutation}, obtained from $\pi$ by underlining   each pair of adjacency letters $(\pi_i,\pi_{i+1})$ whenever the $i$th step of $p$ is $\rightarrow$. For example, for the well-labelled positive path $(p,\pi)$, its cluster-permutation representation is 
$$
\underline{\pi}=2\,\underline{10\,15}\,19\,12\,8\,\underline{9\,1\,7}\,14\,17\,\underline{18\,4}\,\underline{11\,13}\,20\,\underline{3\,5}\,\underline{6\,16},
$$
 which has $7$ clusters of size one, $5$ clusters of size two and $1$ cluster of size three. Any cluster-permutation that represents a well-labelled positive path is called a {\em ballot cluster-permutation}. In what follows, elements in $\PP_n$ are also considered as ballot cluster-permutations. For a cluster $c=\underline{a_1a_2\ldots a_k}$, let $c^r=\underline{a_ka_{k-1}\ldots a_1}$ be the {\em reverse cluster} of $c$.

\begin{theorem}\label{thm:ext}
There exists a bijection $\Phi: \PP_n\rightarrow\C_n$ such that for any $\underline{\pi}\in\PP_n$
\begin{equation}\label{cluster}
\text{$c$ is a cluster in $\underline{\pi}$ } \Longleftrightarrow \text{ either $c$ or $c^r$ is a cluster in $\Phi(\underline{\pi})$}.
\end{equation}
Moreover, $\Phi$ is an extension of $\Psi$, i.e., $\Phi|_{\B_n}=\Psi$.  
\end{theorem}
 Since the construction of $\Phi$ is almost the same as $\Psi$, we will outlined it below and leave the rather routine details to the readers.
 \begin{proof}[{\bf Sketch of the proof of Theorem~\ref{thm:ext}}] A positive path starting at $(0,0)$ and ending at the $x$-axis is called a Motzkin path. A ballot cluster-permutation $(p,\pi)\in\PP_n$ is called a {\em Motzkin permutation} if $p$ is a Motzkin path. 
 
 {\bf The algorithm $\Phi$.} Given a ballot cluster-permutation $\underline{\pi}=(p,\pi)\in\PP_n$, we perform the following three steps to get $\Psi(\pi)\in\Od_n$:
\begin{enumerate}
\item Firstly, we  decompose $\underline{\pi}$ into a pair $(D,\sigma)$, where $D$ is the set of all maximal consecutive subwords of $\underline{\pi}$ that are Motzkin permutations of size at least $2$ and $\sigma$ is the box-permutation obtained from $\pi$ by replacing each maximal Motzkin permutation (retaining its two boundary letters)  in $D$ by a box. 

 For the example of  ballot cluster-permutation 
 $$
\underline{\pi}=2\,\underline{10\,15}\,19\,12\,8\,\underline{9\,1\,7}\,14\,17\,\underline{18\,4}\,\underline{11\,13}\,20\,\underline{3\,5}\,\underline{6\,16},
$$
in Fig.~\ref{positive:paths}, $D$ contains four Motzkin permutations which are painting in red and $\sigma=\blue{2}\,\Box\,\blue{8\,9}\,\Box\,\blue{7\,14\,17\,18}\,\Box\,\blue{4\,11}\,\Box\,\blue{5\,6}\,\Box\,\blue{16}$ is a box-permutation. 
\item Next, using the map $\psi$ in Theorem~\ref{thm:box} to form the cyclic box-permutation $\psi(\sigma)$. 
Continuing with the example, we get the cyclic box-permutation
$$
\psi(\sigma)=(8,2,\Box),(18,17,14,4,\Box),(16,5,\Box,11,9,\Box,7,6,\Box)
$$
 from~\eqref{exm2:box}. 
 \item Finally, inserting each Motzkin permutation in $D$ back to its corresponding box in $\psi(\sigma)$ {\em properly} to form  $\Phi(\underline{\pi})$. 
 Continuing with the above example, after inserting all the four Motzkin permutations back to $\psi(\sigma)$, we get the OCP
 $$
\Phi(\underline{\pi})=(8,2,\red{\underline{10\, 15}, 19,12}),(17,14,\red{\underline{4\,18}}),(\red{\underline{5\,3},20,\underline{13\,11}},\red{\underline{9\,1\,7}},\red{\underline{6\,16}}). 
 $$
 \end{enumerate}
 
 It is clear from the construction that $\Phi$ satisfies property~\eqref{cluster}. The construction of $\Phi^{-1}$ is almost identical to $\Psi^{-1}$ and will be omitted. This proves that $\Phi$ is a bijection. The second statement is clear from the construction of $\Phi$ above. 
 \end{proof} 
 Denote by $\CC_n$  the set of OCPs in $\C_n$ with only one odd cluster-cycle  and by $\CC_{n,2k+1}$ the set of cluster-cycles  in $\CC_n$ with order $2k+1$. 
 \begin{lemma}\label{lem:cyc}
 For any $n\geq1$ and $0\leq k\leq \lfloor (n-1)/2\rfloor$, 
 $$
|\CC_{n,2k+1}|={n\choose 2k+1}(n-1)!.
 $$
 \end{lemma}
 \begin{proof}
 Every odd cluster-cycle  in $\CC_{n,2k+1}$ can be constructed in two steps:
 \begin{enumerate}
 \item arrange the letters $1,2,\ldots n$ in an ordinary cycle;
 \item among the $n$ pairs of adjacency letters, choose $n-2k-1$  pairs and underline them. 
 \end{enumerate}
 There are $(n-1)$ ways to form one ordinary cycle  and ${n\choose n-2k-1}={n\choose 2k+1}$ ways to choose $n-2k-1$  pairs of adjacency letters to underline them, the result then follows. 
 \end{proof}
 
 It follows from Lemma~\ref{lem:cyc} that the exponential generating function for odd cluster-cycles by size and order is
 \begin{align*}
 \sum_{n,k}|\CC_{n,2k+1}|t^{2k+1}\frac{z^n}{n!}&=\sum_{n,k}{n\choose 2k+1}(n-1)!t^{2k+1}\frac{z^n}{n!}\\
 &=\sum_n\frac{z^n}{n}\sum_k{n\choose 2k+1}t^{2k+1}\\
 &=\sum_n\frac{z^n}{n}\frac{(1+t)^n-(1-t)^n}{2}\\
 &=\frac{1}{2}\ln\biggl(\frac{1-z+zt}{1-z-zt}\biggr).
 \end{align*}
 Let us define the {\em order of an OCP} to be the sum of all its cluster-cycles, i.e., the number of clusters in its cluster-cycles. Denote by $\C_{n,k}$ the set of all OCPs in $\C_n$ with order $k$. 
 By the above e.g.f. for odd cluster-cycles and the permutation version of the Compositional Formula~\cite[Corollary~5.1.8]{st2}, the e.g.f. for OCPs by size and order is 
 \begin{equation}\label{egf:ocp}
 1+\sum_{n,k\geq1}|\C_{n,k}|t^{k}\frac{z^n}{n!}=\sqrt{\frac{1-z+zt}{1-z-zt}}.
 \end{equation}
 
 Define the {\em order of a ballot cluster-permutation} by the number of its clusters. Denote by $\PP_{n,k}$ the set of all ballot cluster-permutations (or well-labelled positive paths) in $\PP_n$ with order $k$. An immediate consequence of  Theorem~\ref{thm:ext} and Eq.~\eqref{egf:ocp} is the following e.g.f. formula for well-labelled positive paths by size and order.
 \begin{corollary}
 The e.g.f.  for well-labelled positive paths by size and order is 
  \begin{equation}\label{egf:bcp}
 1+\sum_{n,k\geq1}|\PP_{n,k}|t^{k}\frac{z^n}{n!}=\sqrt{\frac{1-z+zt}{1-z-zt}}.
 \end{equation}
 \end{corollary}
 Note that Eq.~\eqref{egf:bcp} is equivalent to a result of Bernardi, Duplantier and Nadeau~\cite[Corollary~13]{BDN}, which asserts that the number of well-labelled positive paths $(p,\pi)\in\PP_n$ with $p$ having $k$ horizontal steps is
 $$
 \begin{cases}
 {n\choose k}{n-1\choose k}k![(n-k-1)!!]^2\qquad &\text{if $n-k$ is even,}\\
 {n\choose k}{n-1\choose k}k!(n-k)!!(n-k-2)!! &\text{otherwise}.
 \end{cases}
 $$
 
\section{Patterns in ballot permutations}
 \label{sec:4}
 This section deals with the enumerative results presented in Table~\ref{3-pattern}. 
 
 It is known in~\cite[Page~85]{st3} (see also~\cite{Le}) that $123$-avoiding {\em up-down permutations} of length $n$ is counted by the Catalan number $C(\lceil\frac{n}{2}\rceil)$. Since $123$-avoiding ballot permutations are necessarily up-down, we have  $|\B_n(123)|=C(\lceil\frac{n}{2}\rceil)$. 
 
 \begin{theorem}
 For $n\geq2$, 
 $$
 |\B_n(321)|=C(n)-C(n-1)=\frac{3}{n+1}{2n-2\choose n-2}.
 $$
 \end{theorem}
 \begin{proof}
 It is  known that a permutation is $321$-avoiding if and only if both the subsequence formed by its excedance values and the one formed by the remaining non-excedance values are increasing.  Thus, a $321$-avoiding permutation  $\pi\in\S_n$ that is not a ballot permutation in $\B_n(321)$ if and only if $\pi_2=1$. It follows that the map $\pi\mapsto\pi'$, where $\pi'$ is obtained from $\pi$ by removing its second entry and subtracting the other entries by one, sets up an one-to-one correspondence between $\S_n(321)\setminus\B_n(321)$ and $\S_{n-1}(321)$. The result then follows from the fact~\cite{SS} that $|\S_n(321)|=C(n)$. 
 \end{proof}

 \begin{figure}
\begin{tikzpicture}[scale=0.55, axes/.style={black}, thick, shape=circle, inner sep=0.5mm]
\draw[help lines] (0,0) grid (11.5, 5.5);
\path

(0,2)  node[draw=blue, label=left: \textcolor{blue}{4}](4){}
(1,3)  node[draw=blue, label=above:\textcolor{blue}{8}](8){}
(2,2)  node[draw=blue, label=above:\textcolor{blue}{5}](5){}
(3,1)  node[draw=blue, label=above:\textcolor{blue}{2}](2){}
(4,2)  node[draw=blue, label=above:\textcolor{blue}{11}](11){}
(5,1)  node[draw=blue, label=above:\textcolor{blue}{1}](1){}
(6,2)  node[draw=blue, label=above:\textcolor{blue}{6}](6){}
(7,3)  node[draw=blue, label=above:\textcolor{blue}{10}](10){}
(8,2)  node[draw=blue, label=above:\textcolor{blue}{3}](3){}
(9,3)  node[draw=blue, label=above:\textcolor{blue}{7}](7){}
(10,4)  node[draw=blue, label=above:\textcolor{blue}{9}](9){};

\draw[very thick, blue] 
(4) -- (8)--(5)--(2)--(11)--(1)--(6)--(10)--(3)--(7)--(9);

\draw[very thick, dashed, red]  
(9)--(10,0)
(-0.05,1.85)--(-0.05,0);
\begin{scope}[axes]
\draw[->] (0,0) -- (12, 0) node[below=2pt] {$x$} coordinate(x axis);
\draw[->] (0,0) -- (0, 6)  node[left=2pt]  {$y$} coordinate(y axis);
\draw (-0.4,1) node{$h$};
\draw (10.4,2) node{$b$};
\end{scope}
\end{tikzpicture}
\caption{A $(h,b)$-ballot permutation  $2\,8\,5\,4\,11\,1\,6\,10\,3\,7\,9$ with $(h,b)=(2,4)$.
\label{ballot:per}}
\end{figure}
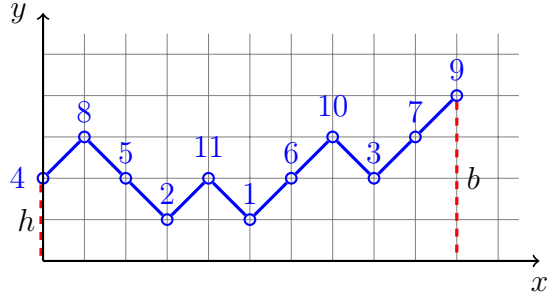
 
 In order to prove Theorem~\ref{thm:213}, we need to consider another generalization of ballot permutations. 
 For two nonnegative integers $h$ and $b$, a permutation $\pi\in\S_{n+1}$ is called a {\em $(h,b)$-ballot permutation} if the lattice path starting at $(0,h)$, whose $i$th step  is $\nearrow$ (resp.~$\searrow$) if $\pi_i<\pi_{i+1}$ (resp.~$\pi_i>\pi_{i+1}$) for each $i\in[n]$,  and ending exactly at $(n,b)$ does  not pass below the $x$-axis.  See Fig.~\ref{ballot:per} for an example of $(2,4)$-ballot permutation. Note that each usual ballot permutation is just a $(0,b)$-ballot permutation for certain $b\geq0$. Let $\B_{n+1}^{(h,b)}$ be the set of all  $(h,b)$-ballot permutations in $\S_{n+1}$. The introduction of $(h,b)$-ballot permutation turns out to be extremely important in our proof of  Theorem~\ref{thm:213}. In fact, the following generalization of Theorem~\ref{thm:213} is true.
 
\begin{theorem}\label{ges:213}
Let $E_n(h,b):=|\B_{n+1}^{(h,b)}(213)|$ and let $F_n(h,b):=|\G(n;h,b)|$. Then,
\begin{equation}\label{iden:ges213}
E_n(h,b)=F_n(h,b)\qquad\text{for any integers $n,h,b\geq0$}. 
\end{equation}
\end{theorem}
We will prove this equinumerosity by showing $E_n(h,b)$ and $F_n(h,b)$ share  the same recurrence relation. 

\begin{lemma}\label{lem:ges}
The number of Gessel walks $F_n(h,b)$ satisfies the recurrence relation:
\begin{equation}\label{ges:wal}
F_n(h,b)=F_{n-1}(h+1,b)+F_{n-1}(h-1,b)+\sum_{0\leq i\leq n-2\atop{0\leq a}}F_i(h+1,a+1)F_{n-2-i}(a,b)
\end{equation}
for $n\geq1$ and $h,b\geq0$ with initial condition $F_0(h,b)=\chi(h=b=0)$.
\end{lemma}
\begin{proof}
Consider the first time after the starting point that a Gessel walk $G\in\G(n;h,b)$ arrives the $y$-axis: 
\begin{enumerate}
\item The first step of $G$ is $\uparrow$. The number of such walks is $F_{n-1}(h+1,b)$.
\item The first step of $G$ is $\downarrow$. The number of such walks is $F_{n-1}(h-1,b)$.
\item The first step of $G$ is $\nearrow$. For fixed integers $a,i\geq0$, the number of such walks that first arrive the $y$-axis at $(0,a)$ after $i+2$ steps is $F_i(h+1,a+1)F_{n-2-i}(a,b)$.
\end{enumerate}
Summing over all the above three cases gives~\eqref{ges:wal}.
\end{proof}

For integers $1\leq k\leq n$, let $[k,n]:=\{k,k+1,\ldots,n\}$. Lemma~\ref{lem:ges} together with the following recursion for $E_n(h,b)$ proves Theorem~\ref{ges:213}. 
\begin{lemma}\label{lema:213}
The number of permutations  $E_n(h,b)$ satisfies the recurrence relation:
\begin{equation}\label{213:bal}
E_n(h,b)=E_{n-1}(h+1,b)+E_{n-1}(h-1,b)+\sum_{0\leq i\leq n-2\atop{0\leq a}}E_i(h+1,a+1)E_{n-2-i}(a,b)
\end{equation}
for $n\geq1$ and $h,b\geq0$ with initial condition $E_0(h,b)=\chi(h=b=0)$.
\end{lemma}
\begin{proof}
Any $213$-avoiding permutation $\pi\in\S_{n+1}(213)$ with $\pi_1=k$ can be decomposed into $k\pi'\pi''$, where $\pi'$ is a $213$-avoiding permutation of $[k+1,n+1]$ and $\pi''$ is a $213$-avoiding permutation of $[k-1]$. Now consider the first letter of a permutation $\pi\in\B_{n+1}^{(h,b)}(213)$:
\begin{enumerate}
\item If $\pi_1=1$, then $\pi=1\pi'$ with $\pi'$ a $213$-avoiding $(h+1,b)$-ballot permutation of  $[2,n+1]$. There are $E_{n-1}(h+1,b)$ such permutations arising in this case. 
\item If $\pi_1=n+1$, then $\pi=(n+1)\pi''$ with $\pi''$ a $213$-avoiding $(h-1,b)$-ballot permutation of  $[n]$. There are $E_{n-1}(h-1,b)$ such permutations arising in this case. 
\item If $\pi_1=n-i$ for some $0\leq i\leq n-2$, then $\pi=(n-i)\pi'\pi''$, where $\pi'$  is a $213$-avoiding permutation of $[n-i+1,n+1]$ and $\pi''$ is a $213$-avoiding permutation of $[n-i-1]$. Moreover, if $\pi'$ is a $(h+1,a+1)$-ballot permutation for some $a\geq0$, then $\pi''$ is a $(a,b)$-ballot permutation. Thus, there are $E_{i}(h+1,a+1)E_{n-2-i}(a,b)$ such permutations arising in this case for any fixed integer $a\geq0$. 
\end{enumerate}
Summing over all the above three cases gives~\eqref{213:bal}. 
\end{proof}

We have two remarks concerning our approach to~\eqref{iden:ges213}.  
\begin{remark}
The proofs of Lemmas~\ref{lem:ges} and~\ref{lema:213} actually  induce a recursive bijection between $\G(n;h,b)$ and $\B_{n+1}^{(h,b)}(213)$. 
\end{remark}

\begin{remark}
By considering the position of the letter $1$, any $\pi\in\S_{n+1}(213)$ can be decomposed as $\pi=\pi'1\pi''$, where $\pi'$ and $\pi''$ are both $213$-avoiding and all letters in $\pi'$ are greater than that in $\pi''$. This consideration leads to another recursion for $E_n(h,b)$ essentially different with~\eqref{213:bal}:  
\begin{equation}
E_n(h,b)=E_{n-1}(h+1,b)+E_{n-1}(h,b+1)+\sum_{0\leq i\leq n-2\atop{0\leq a}}E_i(h,a+1)E_{n-2-i}(a+1,b).
\end{equation}
We failed to find a parallel decomposition of Gessel walks that could lead to the same recursion as above for $F_n(h,b)$. It remains mysterious to find such a decomposition for Gessel walks. 
\end{remark}

An interesting consequence of Theorem~\ref{ges:213} is a new interpretation of $g_n$. 

\begin{corollary}
The number of $213$-avoiding Dyck permutations of length $n$ is 
$$
g_n=16^n\frac{(5/6)_n(1/2)_n}{(5/3)_n(2)_n},
$$
which enumerates $2n$-step Gessel walks that starting and ending at $(0, 0)$.
\end{corollary}

Regarding $231$-avoiding ballot permutations, we have a connection to another class of lattice walks  that was studied by Gouyou-Beauchamps in~\cite{GB}. A {\em Gouyou-Beauchamps walk} ({\em GB walk} for short) is a lattice path  confined to the quarter plane $\N^2$ using steps from the set $\{\uparrow,\downarrow,\leftarrow,\rightarrow\}$ and lying weakly below the main diagonal $y=x$. See Fig.~\ref{G:GBwalks} (right graph) for an example of GB walk. A  nice enumerative result due to Gouyou-Beauchamps~\cite{GB} asserts that the number of $n$-step GB walks that starting  at $(0, 0)$ and ending at $x$-axis is the product of Catalan numbers 
$$
C(\lfloor (n+1)/2\rfloor)C(\lfloor(n+2)/2\rfloor).
$$
This integer sequence appears as A005817 in the OEIS~\cite{oeis}, where several other intriguting combinatorial interpretations are known. 

For any $h,b\in\N$, let $\Hw(n;h,b)$ be the set of  $n$-step GB walks that starting at $(h,0)$ and ending at $(b,0)$. The following result for $231$-avoiding ballot permutations is parallel to Theorem~\ref{ges:213}.

\begin{theorem}\label{gb:231}
Let $G_n(h,b):=|\B_{n+1}^{(h,b)}(231)|$ and let $H_n(h,b):=|\Hw(n;h,b)|$. Then,
\begin{equation}\label{iden:231}
G_n(h,b)=H_n(h,b)\qquad\text{for any integers $n,h,b\geq0$}. 
\end{equation}
In particular, $|\B_{n}(231)|=C(\lfloor n/2\rfloor)C(\lfloor(n+1)/2\rfloor)$.
\end{theorem}

We will prove Theorem~\ref{gb:231} by showing  $G_n(h,b)$ and $H_n(h,b)$ share the same recurrence relation. 

\begin{lemma}\label{lem:GB}
The number of GB walks $H_n(h,b)$ satisfies the recurrence relation:
\begin{equation}\label{GB:wal}
H_n(h,b)=H_{n-1}(h+1,b)+H_{n-1}(h-1,b)+\sum_{0\leq i\leq n-2\atop{0\leq a}}H_i(h-1,a)H_{n-2-i}(a+1,b)
\end{equation}
for $n\geq1$ and $h,b\geq0$ with initial condition $H_0(h,b)=\chi(h=b=0)$.
\end{lemma}
\begin{proof}
Consider the first time after the starting point that a GB walk $H\in\Hw(n;h,b)$ visits the $x$-axis: 
\begin{enumerate}
\item The first step of $H$ is $\rightarrow$. The number of such walks is $H_{n-1}(h+1,b)$.
\item The first step of $H$ is $\leftarrow$. The number of such walks is $H_{n-1}(h-1,b)$.
\item The first step of $H$ is $\uparrow$. For fixed integers $a,i\geq0$, the number of such walks that first return the $x$-axis at $(a+1,0)$ after $i+2$ steps is $H_i(h-1,a)H_{n-2-i}(a+1,b)$.
\end{enumerate}
Summing over all the above three cases gives~\eqref{GB:wal}.
\end{proof}

In view of Lemma~\ref{lem:GB}, the following recurrence relation for $G_n(h,b)$ proves Theorem~\ref{gb:231}. 
\begin{lemma}\label{lema:231}
The number of permutations  $G_n(h,b)$ satisfies the recurrence relation:
\begin{equation}\label{231:bal}
G_n(h,b)=G_{n-1}(h+1,b)+G_{n-1}(h-1,b)+\sum_{0\leq i\leq n-2\atop{0\leq a}}G_i(h-1,a)G_{n-2-i}(a+1,b)
\end{equation}
for $n\geq1$ and $h,b\geq0$ with initial condition $G_0(h,b)=\chi(h=b=0)$.
\end{lemma}
\begin{proof}
Any $231$-avoiding permutation $\pi\in\S_{n+1}(231)$ with $\pi_1=k$ can be decomposed into $k\pi'\pi''$, where $\pi'$ is a $231$-avoiding permutation of $[k-1]$ and $\pi''$ is a $231$-avoiding permutation of $[k+1,n+1]$. Now consider the first letter of a permutation $\pi\in\B_{n+1}^{(h,b)}(231)$:
\begin{enumerate}
\item If $\pi_1=1$, then $\pi=1\pi'$ with $\pi'$ a $231$-avoiding $(h+1,b)$-ballot permutation of  $[2,n+1]$. There are $G_{n-1}(h+1,b)$ such permutations arising in this case. 
\item If $\pi_1=n+1$, then $\pi=(n+1)\pi''$ with $\pi''$ a $231$-avoiding $(h-1,b)$-ballot permutation of  $[n]$. There are $G_{n-1}(h-1,b)$ such permutations arising in this case. 
\item If $\pi_1=i+2$ for some $0\leq i\leq n-2$, then $\pi=(i+2)\pi'\pi''$, where $\pi'$  is a $231$-avoiding permutation of $[i+1]$ and $\pi''$ is a $231$-avoiding permutation of $[i+3,n+1]$. Moreover, if $\pi'$ is a $(h-1,a)$-ballot permutation for some $a\geq0$, then $\pi''$ is a $(a+1,b)$-ballot permutation. Thus, there are $G_{i}(h-1,a)G_{n-2-i}(a+1,b)$ such permutations arising in this case for any fixed integer $a\geq0$. 
\end{enumerate}
Summing over all the above three cases gives~\eqref{231:bal}. 
\end{proof}

In Table~\ref{3-pattern}, we still need to prove the Wilf-equivalences of $213$ and $312$ and of $132$ and $231$ over ballot permutations. They are direct consequence of two simple recursive bijections $\varphi: \S_n(213)\rightarrow\S_n(312)$ and $\eta: \S_n(132)\rightarrow\S_n(231)$ that preserve the positions of descents:
\begin{itemize}
\item If $\pi\in\S_n(213)$, then $\pi=\pi'1\pi''$. Define $\varphi(\pi)=\varphi(\pi'')1\varphi(\pi')$. 
\item If $\pi\in\S_n(132)$, then $\pi=\pi'n\pi''$. Define $\eta(\pi)=\eta(\pi'')n\eta(\pi')$.
\end{itemize}

\section{Final remarks, open problems}
\label{sec:5}

This paper focus mainly on the enumerative aspect of ballot permutations. Two main achievements are:
\begin{enumerate}
\item A Foata-style bijection between ballot permutations and odd order permutations that not only proves bijectively  two conjectures uniformly but also can be extended  perfectly to well-labelled positive paths, a class of generalized ballot permutations arising from  polytope theory. 
\item Complete the enumeration of ballot permutations avoiding a pattern of length $3$. Connections with Gessel walks and Gouyou-Beauchamps walks are established. 
\end{enumerate}

Our work can be extended in several directions. For $\m=(m_1,\ldots,m_n)\in\N^n$, let $\S_{\m}$ be the set of multipermutations of $\{1^{m_1},2^{m_2},\ldots,n^{m_n}\}$. An element $\pi\in\S_{\m}$ is called a {\em ballot multipermutations} if for each $i$, $1\leq i<\sum_{i=1}^nm_n$, there holds
$$
|\{j\in[i]: \pi_j<\pi_{j+1}\}|\geq |\{j\in[i]: \pi_j>\pi_{j+1}\}|.
$$

\begin{?}
For fixed $\m\in\N^n$, can the ballot multipermutations in $\S_{\m}$ be enumerated? 
\end{?}
For classical patterns in ballot permutations, one further direction to continue would be to consider a pair of length-$3$ patterns or a single pattern of length $4$.  Moreover, it would also be interesting to explore systematically ballot permutations avoid consecutive patterns or the more general vincular patterns~\cite{kit}. 

For any $i,j\in\N$, let $g_n(i,j)$ be the number of  $n$-step Gessel walks that starting at $(0,0)$ and ending at $(i, j)$. Bostan and  Kauers~\cite{BK} proved that the  trivariate generating function 
$\sum_{n,i,j\geq0} g_n(i,i)t^nx^iy^j$ is algebraic. In view of their result, we pose the following problem for further research. 

\begin{?}
Are the generating functions for Gessel walks and Gouyou-Beauchamps walks
$$
F(t,x,y)=\sum_{n,h,b\geq0} F_n(h,b)t^nx^hy^b\quad\text{and}\quad H(t,x,y)=\sum_{n,h,b\geq0} H_n(h,b)t^nx^hy^b
$$
algebraic? 
\end{?}

\section*{Acknowledgement}

Lin was supported
by the National Science Foundation of China grant 11871247 and the project of Qilu Young Scholars of Shandong University. Wang  was supported by the National Science Foundation of China grant 11671037. Zhao was supported by the Science Foundation of China University of Petroleum, Beijing (Grant No. 2462020YXZZ004).


\begin{thebibliography}{99}



\bibitem{BDN} O. Bernardi, B. Duplantier and P. Nadeau, A bijection between well-labelled positive paths and matchings, S\'em. Lothar. Combin., {\bf63} (2010), Article B63e.

\bibitem{BS}H. Bidkhori and S. Sullivant, Eulerian-catalan numbers, Electron. J. Combin., {\bf18}  (2011), \#P187.

\bibitem{BK} A. Bostan and M. Kauers,  The complete generating function for Gessel walks is algebraic (With an appendix by Mark van Hoeij), Proc. Amer. Math. Soc., {\bf138} (2010), 3063--3078.

\bibitem{BKR} A. Bostan, I. Kurkova, K. Raschel, A human proof of Gessel's lattice path conjecture, Trans. Amer. Math. Soc.,  {\bf369} (2017), 1365--1393. 

\bibitem{BM2016} M. Bousquet-M\'elou, An elementary solution of Gessel's walks in the quadrant, Adv. Math., {\bf303} (2016), 1171--1189.

\bibitem{BMM} M. Bousquet-M\'elou and M. Mishna,  Walks with small steps in the quarter plane, Algorithmic probability and combinatorics, 1--39, Contemp. Math., 520, Amer. Math. Soc., Providence, RI, 2010. 



\bibitem{ges}  I.M. Gessel, A factorization for formal Laurent series and lattice path enumeration, J. Combin.
Theory Ser. A, {\bf28} (1980), 321--337.

\bibitem{GB} D. Gouyou-Beauchamps, Chemins sous-diagonaux et tableaux de Young, pp. 112--125 of "{\em Combinatoire \'enum\'erative} (Montreal 1985)", Lect. Notes Math. 1234, 1986.

\bibitem{KKZ} M. Kauers, C. Koutschan and D. Zeilberger, Proof of Ira Gessel's lattice path conjecture, Proc. Natl. Acad. Sci. USA, {\bf106} (2009) 11502--11505.

\bibitem{kit} S.~Kitaev, {\em Patterns in permutations and words}, Springer Science \& Business Media, 2011. 


\bibitem{Le} J.B. Lewis,
Pattern avoidance for alternating permutations and Young tableaux, 
J. Combin. Theory Ser. A, {\bf118} (2011), 1436--1450.

\bibitem{lot} M. Lothaire, {\em Combinatorics on words}, Cambridge University Press, 1997 (Encyclopedia of Math. and its Appl., v. {\bf17}).

\bibitem{oeis} OEIS Foundation Inc., The On-Line Encyclopedia of Integer Sequences,  \href{http://oeis.org}{http://oeis.org}, 2020.

\bibitem{SS} R.~Simion, F.~Schmidt, Restricted permutations, European J. Combin. {\bf6} (1985), 383--406.

\bibitem{spi} S. Spiro, Ballot permutations and odd order permutations, Discrete Math., {\bf343} (2020), 111869.

\bibitem{st1} R.P. Stanley, {\em Enumerative Combinatorics, vol. 1}, Cambridge Stud. Adv. Math., vol. 49, Cambridge University Press, Cambridge, 1997.

\bibitem{st2} R.P. Stanley, {\em Enumerative Combinatorics, vol. 2}, Cambridge Stud. Adv. Math., vol. 62, Cambridge University Press, Cambridge, 1999.

\bibitem{st3} R.P. Stanley,
{\em Catalan numbers}, Cambridge University Press, New York, 2015.

\bibitem{SZ} Y. Sun and T. Zhao, Refined Eulerian numbers and ballot permutations, in preparation. 

\bibitem{WZ} D.G.L. Wang and J.J.R. Zhang, A Toeplitz property of ballot permutations and odd order permutations, Electron. J. Combin. 27(2) (2020), P2.55.

\bibitem{zhao} D.G.L. Wang and T. Zhao, The peak and descent statistics over ballot permutations, \href{https://arxiv.org/abs/2009.05973v2}{arXiv: 2009.05973v2.}




\end{thebibliography}
\end{document}